\documentclass[a4paper,10pt,english]{smfart}
\usepackage[english]{babel}
\usepackage[T1]{fontenc}
\usepackage{lmodern}
\usepackage{graphicx}
\usepackage{tabularx}
\usepackage{amsmath,amsthm,amssymb,amsfonts}
\usepackage[a4paper,left=4cm,right=4cm,top=4cm,bottom=4cm]{geometry}
\numberwithin{equation}{section}
\title[Density of rational points on a certain bihomogeneous threefold]
{Density of rational points on a certain smooth bihomogeneous threefold}
\author{Pierre Le Boudec}
\subjclass{$11$D$45$, $14$G$05$}
\keywords{Bihomogeneous varieties, rational points, Manin's conjecture}
\address{Institute for Advanced Study \\ School of Mathematics \\ Einstein Drive \\
Simonyi Hall $-$ \text{Office $111$} \\ Princeton, NJ $08540$ \\ USA}
\email{pleboudec@ias.edu}

\begin{document}

\makeatletter
\def\imod#1{\allowbreak\mkern10mu({\operator@font mod}\,\,#1)}
\makeatother

\newtheorem{lemma}{Lemma}
\newtheorem{theorem}{Theorem}
\newtheorem{corollary}{Corollary}
\newtheorem{proposition}{Proposition}
\newtheorem{conjecture}{Conjecture}

\newcommand{\vol}{\operatorname{vol}}
\newcommand{\D}{\mathrm{d}}
\newcommand{\rank}{\operatorname{rank}}
\newcommand{\Pic}{\operatorname{Pic}}
\newcommand{\Gal}{\operatorname{Gal}}
\newcommand{\meas}{\operatorname{meas}}
\newcommand{\Spec}{\operatorname{Spec}}
\newcommand{\eff}{\operatorname{eff}}
\newcommand{\rad}{\operatorname{rad}}
\newcommand{\sq}{\operatorname{sq}}
\newcommand{\sign}{\operatorname{sign}}

\begin{abstract}
We establish sharp upper and lower bounds for the number of rational points of bounded anticanonical height on a smooth bihomogeneous
threefold defined over $\mathbb{Q}$ and of bidegree $(1,2)$. These bounds are in agreement with Manin's conjecture.
\end{abstract}

\maketitle

\tableofcontents

\section{Introduction}

Let $n \geq 2$ and $d \geq 1$ be two integers such that $n \geq d$. Let $V_{d}^n \subset \mathbb{P}^{n} \times \mathbb{P}^{n}$
be the smooth hypersurface defined over a number field $K$ by the equation
\begin{equation*}
x_0 y_0^d + \cdots + x_n y_n^d = 0,
\end{equation*}
where we use the notation $(\mathbf{x}, \mathbf{y}) = ((x_0 : \dots : x_n), (y_0 : \dots : y_n))$ to denote the coordinates in the
biprojective space $\mathbb{P}^{n} \times \mathbb{P}^{n}$.

The family of smooth bihomogeneous varieties $V_{d}^n$ is an excellent testing ground for the validity of Manin's conjecture on
the asymptotic behaviour of the number of rational points of bounded anticanonical height on Fano varieties (see \cite{MR974910}).
For instance, Batyrev and Tschinkel have provided a famous counterexample to this conjecture in the case $n = 3$, $d = 3$, and under
the assumption that $K$ contains a nontrivial cube root of unity.

From now on, we focus on the case $K = \mathbb{Q}$. We define the usual exponential height function
$H : \mathbb{P}^{n}(\mathbb{Q}) \to \mathbb{R}_{> 0}$ as follows. Given $\mathbf{z} \in \mathbb{P}^n(\mathbb{Q})$, we can choose
coordinates $(z_0 : \dots : z_n)$ satisfying $(z_0, \dots, z_n) \in \mathbb{Z}^{n+1}$ and $\gcd(z_0, \dots, z_n) = 1$, and then
we can set
\begin{equation*}
H(\mathbf{z}) = \max \{ |z_i|, i = 0, \dots, n \}.
\end{equation*}
With this in mind, we can define a height function
$\mathbf{H} : \mathbb{P}^{n}(\mathbb{Q}) \times \mathbb{P}^{n}(\mathbb{Q}) \to \mathbb{R}_{> 0}$ by setting
\begin{equation*}
\mathbf{H}(\mathbf{x}, \mathbf{y}) = H(\mathbf{x})^{n} H(\mathbf{y})^{n + 1 - d},
\end{equation*}
for $(\mathbf{x}, \mathbf{y}) \in \mathbb{P}^{n}(\mathbb{Q}) \times \mathbb{P}^{n}(\mathbb{Q})$. For any Zariski open subset
$U_{d}^n$ of $V_{d}^n$, we can introduce the number of rational points of bounded anticanonical height on $U_{d}^n$, that is
\begin{equation*}
N_{U_{d}^n, \mathbf{H}}(B) = \#
\{ (\mathbf{x}, \mathbf{y}) \in U_{d}^n(\mathbb{Q}), \mathbf{H}(\mathbf{x}, \mathbf{y}) \leq B \}.
\end{equation*}
In this setting, Manin's conjecture predicts that there should exist an open subset $U_{d}^n$ of $V_{d}^n$ such that
\begin{equation}
\label{Manin conjecture}
N_{U_{d}^n, \mathbf{H}}(B) = c B \log B (1 + o(1)),
\end{equation}
where $c > 0$ is a constant depending on $V_{d}^n$ and $\mathbf{H}$, and which is expected to obey Peyre's prediction
\cite{MR1340296}. As already mentioned, this conjecture is known not to hold in such generality.

The circle method is a traditional technique to count solutions to diophantine equations, and it has recently been applied by
Schindler \cite{Schindler, Schindler2} to count rational points on bihomogeneous varieties. However, this method is only
expected to yield a proof of Manin's conjecture for $V_{d}^n$ if $n$ is exponentially large in terms of~$d$. As a result, it is not
reasonable to ask for a proof of Manin's conjecture for any $d \geq 1$ and any $n \geq d + 1$. It is thus natural to start by
investigating the cases where $d$ is small. In particular, for fixed $d$, we would like to find out how small $n$ needs to be in
terms of $d$, to allow us to approach Manin's conjecture.

Let us note that no counterexample to Manin's conjecture is known if $d \leq 2$, and the conjecture is expected to hold for any
$n \geq 2$ in this case.

If $d = 1$ then the problem has been settled for any $n \geq 2$ by a great variety of techniques. First, the asymptotic formula
\eqref{Manin conjecture} follows from the result of Franke, Manin and Tschinkel \cite{MR974910} on flag varieties, which makes use of
the work of Langlands about the meromorphic continuation of Eisenstein series. Then, the result has also been obtained by Thunder
\cite{MR1237919} using the geometry of numbers. Finally, it has also been established using the circle method by Robbiani
\cite{MR1801715} for any $n \geq 3$, and more recently by Spencer \cite{MR2521490} for any $n \geq 2$.

The next case to study is $d = 2$. Here, the circle method is expected to establish the conjectured asymptotic formula
\eqref{Manin conjecture} provided that $n \geq 4$. The cases $n \in \{2, 3 \}$ are open and are known to be extremely hard problems.
The circle method might eventually succeed if $n = 3$ since this case seems to be at the border of the scope of the method. However,
the case $n = 2$ is far out of reach, and the aim of this article is to investigate what can be achieved in this case.

Unfortunately, we are unable to establish Manin's conjecture for $V_{2}^2$. However, we are able to prove upper and lower bounds
of the exact order of magnitude for $N_{U_{2}^2, \mathbf{H}}(B)$, where $U_{2}^2$ is the open subset defined by removing from
$V_{2}^2$ the subset given by $x_0 x_1 x_2 y_0 y_1 y_2 = 0$.

Our main result is the following.

\begin{theorem}
\label{Main Theorem}
We have the bounds
\begin{equation*}
B \log B \ll  N_{U_{2}^2, \mathbf{H}}(B) \ll B \log B.
\end{equation*}
\end{theorem}

We note that this result is in agreement with Manin's prediction \eqref{Manin conjecture}. Therefore, it is likely to constitute
a first step in the direction of a proof of Manin's conjecture for $V_{2}^2$.

Let us give a sketch of the proof. In what follows, we denote by $\varphi_i : \mathbb{P}^2 \times \mathbb{P}^2 \to \mathbb{P}^2$,
$i \in \{1, 2 \}$, the two projections.

First, we note that proving the lower bound is not hard since it suffices to note that the contribution to
$N_{U_{2}^2, \mathbf{H}}(B)$ of the fibers of $\varphi_2$ corresponding to rational points $\mathbf{y} \in \mathbb{P}^2(\mathbb{Q})$
whose height is bounded by a small power of $B$ is of the expected order of magnitude. This is achieved in section \ref{lower bound}.

The proof of the upper bound is more intricate. It mainly relies on lemma \ref{geometry lemma 4} which gives an upper bound for the
number of solutions to a slightly more general equation than $x_0 y_0^2 + x_1 y_1^2 + x_2 y_2^2 = 0$. To prove this lemma,
we make use of both geometry of numbers and analytic number theory results.

More specifically, we get a first upper bound by estimating the number of $\mathbf{x} \in \mathbb{P}^2(\mathbb{Q})$ for fixed
$\mathbf{y} \in \mathbb{P}^2(\mathbb{Q})$ and by summing trivially over the fibers of $\varphi_2$. Similarly, we obtain a second upper
bound by estimating the number of $\mathbf{y} \in \mathbb{P}^2(\mathbb{Q})$ for fixed $\mathbf{x} \in \mathbb{P}^2(\mathbb{Q})$.
However, it is worth noticing that the summation over the fibers of $\varphi_1$ has to be carried out non-trivially because we
need to take advantage of the fact that most diagonal conics do not have a rational point. To complete the proof, it only remains
to minimize these two upper bounds, basically depending on the respective sizes of $\mathbf{x}$ and $\mathbf{y}$.

It is worth emphasizing the fact that the equation studied in lemma \ref{geometry lemma 4} shows up in various other settings. As a
consequence, lemma \ref{geometry lemma 4} is likely to be very useful in other situations. For instance, it plays a crucial role in
the work of the author \cite{Elliptic}, where it is proved that certain elliptic fibrations have linear growth, as predicted by
Manin's conjecture.

\subsection{Acknowledgements}

It is a pleasure for the author to thank Tim Browning for interesting explanations and comments on the results given in
section \ref{Geometry section}, and Damaris Schindler for kindly answering questions about her works, and for comments on an earlier
draft of this paper.

The financial support and the perfect working conditions provided by the Institute for Advanced Study are gratefully acknowledged.
This material is based upon work supported by the National Science Foundation under agreement No. DMS-$1128155$. Any opinions,
findings and conclusions or recommendations expressed in this material are those of the author and do not necessarily reflect
the views of the National Science Foundation.

\section{Geometry of numbers}

\label{Geometry section}

We now recall two lemmas which provide upper bounds for the number of solutions to certain homogeneous diagonal equations
in three variables and constrained in boxes. The first of these two lemmas deals with the case of a linear equation and is due to
Heath-Brown \cite[Lemma~$3$]{MR757475}.

\begin{lemma}
\label{geometry lemma 1}
Let $\mathbf{w} = (w_0,w_1,w_2) \in \mathbb{Z}^3$ be a primitive vector and let $U_i \geq 1$ for $i \in \{0, 1, 2 \}$. Let also
$N_{\mathbf{w}} = N_{\mathbf{w}}(U_0,U_1,U_2)$ be the number of primitive vectors $(u_0,u_1,u_2) \in \mathbb{Z}^3$
satisfying $|u_i| \leq U_i$ for $i \in \{ 0, 1, 2 \}$ and the equation
\begin{equation*}
u_0 w_0 + u_1 w_1 + u_2 w_2 = 0.
\end{equation*}
We have the bound
\begin{equation*}
N_{\mathbf{w}} \leq 12 \pi \frac{U_0 U_1 U_2}{\max \{ |w_i| U_i \} } + 4,
\end{equation*}
where the maximum is taken over $i \in \{ 0, 1, 2 \}$. In particular, if $\mathbf{w} \in \mathbb{Z}_{\neq 0}^3$ then
\begin{equation*}
N_{\mathbf{w}} \ll \frac{(U_0 U_1 U_2)^{2/3}}{ |w_0 w_1 w_2|^{1/3}} + 1.
\end{equation*}
\end{lemma}

The second lemma is concerned with the case of a quadratic equation and is a particular case of a result of Browning
\cite[Lemma $4$.$10$]{MR2559866}. Let us mention that due to a subtle oversight in the proof of \cite[Lemma $4$.$9$]{MR2559866},
one should replace the arithmetic function $2^{\omega}$ by $\tau$ in the statement of \cite[Lemma $4$.$10$]{MR2559866} (see the
recent result of Browning and Swarbrick Jones \cite[Theorem $5$]{SJ}).

\begin{lemma}
\label{geometry lemma 2}
Let $\mathbf{u} = (u_0,u_1,u_2) \in \mathbb{Z}_{\neq 0}^3$ be a vector satisfying the conditions $\gcd(u_i,u_j) = 1$ for
$i,j \in \{0, 1, 2 \}$, $i \neq j$, and let $V_i \geq 1$ for $i \in \{0, 1, 2 \}$. Let also
$N_{\mathbf{u}} = N_{\mathbf{u}}(V_0,V_1,V_2)$ be the number of primitive vectors $(v_0,v_1,v_2) \in \mathbb{Z}^3$ satisfying
$|v_i| \leq V_i$ for $i \in \{ 0, 1, 2 \}$ and the equation
\begin{equation*}
u_0 v_0^2 + u_1 v_1^2 + u_2 v_2^2 = 0.
\end{equation*}
We have the bound
\begin{equation*}
N_{\mathbf{u}} \ll \left( \frac{V_0 V_1 V_2}{ |u_0 u_1 u_2| } + 1 \right)^{1/3} \tau(|u_0 u_1 u_2|).
\end{equation*}
\end{lemma}

We also need to consider how often a diagonal quadratic equation has a non-trivial integral solution. For this, we recall the
following lemma, which is a particular case of the nice result of Browning \cite[Proposition $1$]{MR2250046}. Let us note that this
result is deep and builds upon several powerful analytic number theory tools.

\begin{lemma}
\label{geometry lemma 3}
Let $\mathbf{f} = (f_0,f_1,f_2) \in \mathbb{Z}_{\neq 0}^3$ be a primitive vector and let $U_i \geq 1$ for $i \in \{ 0, 1, 2 \}$.
Let also $\mathcal{T}_{\mathbf{f}}(U_0, U_1, U_2)$ be the set of $\mathbf{u} = (u_0,u_1,u_2) \in \mathbb{Z}_{\neq 0}^3$ satisfying
$|u_i| \leq U_i$ for $i \in \{ 0, 1, 2 \}$, and $\gcd(u_i, u_j) = 1$ for $i, j \in \{0, 1, 2 \}$, $i \neq j$, and such that the
equation
\begin{equation*}
f_0 u_0 v_0^2 + f_1 u_1 v_1^2 + f_2 u_2 v_2^2 = 0,
\end{equation*}
has a solution $(v_0,v_1,v_2) \in \mathbb{Z}_{\neq 0}^3$ with $\gcd(v_i, v_j) = 1$ for $i, j \in \{0, 1, 2 \}$, $i \neq j$.
Let $\varepsilon > 0$ be fixed. We have the bound
\begin{equation*}
\sum_{\mathbf{u} \in \mathcal{T}_{\mathbf{f}}(U_0, U_1, U_2)} 2^{\omega(u_0 u_1 u_2)} \ll
|f_0 f_1 f_2|^{\varepsilon} U_0 U_1 U_2 M_{\varepsilon}(U_0, U_1, U_2),
\end{equation*}
where
\begin{equation*}
M_{\varepsilon}(U_0, U_1, U_2) = 1 + \max_{\{ i, j, k \} = \{ 0, 1, 2 \} } (U_i U_j)^{-1/2 + \varepsilon} \log 2U_k.
\end{equation*}
\end{lemma}

These three lemmas together allow us to prove a sharp upper bound for the number of solutions
$(\mathbf{u}, \mathbf{v}) \in \mathbb{Z}_{\neq 0}^3 \times \mathbb{Z}_{\neq 0}^3$ to the equation of lemma \ref{geometry lemma 3}
and constrained in boxes. More precisely, we establish the following lemma, which is the key result in the proof of the
upper bound in Theorem \ref{Main Theorem}.

\begin{lemma}
\label{geometry lemma 4}
Let $\mathbf{f} = (f_0,f_1,f_2) \in \mathbb{Z}_{\neq 0}^3$ be a vector satisfying the conditions $\gcd(f_i, f_j) = 1$ for
$i, j \in \{0, 1, 2 \}$, $i \neq j$, and let $U_i, V_i \geq 1$ for $i \in \{ 0, 1, 2 \}$. Let also
$N_{\mathbf{f}} = N_{\mathbf{f}}(U_0,U_1,U_2,V_0,V_1,V_2)$ be the number of vectors $(u_0,u_1,u_2) \in \mathbb{Z}_{\neq 0}^3$ and
$(v_0,v_1,v_2) \in \mathbb{Z}_{\neq 0}^3$ satisfying $|u_i| \leq U_i$, $|v_i| \leq V_i$ for $i \in \{ 0, 1, 2 \}$, and the equation
\begin{equation*}
f_0 u_0 v_0^2 + f_1 u_1 v_1^2 + f_2 u_2 v_2^2 = 0,
\end{equation*}
and such that $\gcd(u_i v_i, u_j v_j) = 1$ for $i, j \in \{0, 1, 2 \}$, $i \neq j$. Let $\varepsilon > 0$ be fixed and recall the
definition of $M_{\varepsilon}(U_0, U_1, U_2)$ given in lemma \ref{geometry lemma 3}. We have the bound
\begin{equation*}
N_{\mathbf{f}} \ll |f_0 f_1 f_2|^{\varepsilon} (U_0 U_1 U_2)^{2/3} (V_0 V_1 V_2)^{1/3} M_{\varepsilon}(U_0, U_1, U_2).
\end{equation*}
\end{lemma}

\begin{proof}
First, let us fix $(v_0,v_1,v_2) \in \mathbb{Z}_{\neq 0}^3$ and let us start by bounding the number of
$(u_0,u_1,u_2) \in \mathbb{Z}_{\neq 0}^3$ satisfying the conditions stated in the lemma. Since
$\gcd(f_0 v_0^2, f_1 v_1^2, f_2 v_2^2) = 1$, lemma~\ref{geometry lemma 1} gives
\begin{equation*}
N_{\mathbf{f}} \ll \sum_{\substack{|v_i| \leq V_i \\ i \in \{ 0, 1, 2 \}}}
\left( \frac1{|f_0 f_1 f_2|^{1/3}} \frac{(U_0 U_1 U_2)^{2/3}}{|v_0 v_1 v_2|^{2/3}} + 1 \right).
\end{equation*}
In particular, this gives us a first upper bound
\begin{equation}
\label{bound 1}
N_{\mathbf{f}} \ll (U_0 U_1 U_2)^{2/3} (V_0 V_1 V_2)^{1/3} + V_0 V_1 V_2.
\end{equation}
In a similar fashion, let us fix $(u_0,u_1,u_2) \in \mathbb{Z}_{\neq 0}^3$ and let us start by bounding the number of
$(v_0,v_1,v_2) \in \mathbb{Z}_{\neq 0}^3$ satisfying the conditions stated in the lemma. The equation
\begin{equation*}
f_0 u_0 v_0^2 + f_1 u_1 v_1^2 + f_2 u_2 v_2^2 = 0,
\end{equation*}
and the coprimality conditions $\gcd(f_i, f_j) = \gcd(u_i v_i, u_j v_j) = 1$ for $i, j \in \{0, 1, 2 \}$, $i \neq j$, imply
that $\gcd(f_i u_i, f_j u_j) = 1$ for $i, j \in \{0, 1, 2 \}$, $i \neq j$. We can thus apply lemma \ref{geometry lemma 2}. Recalling
the notation introduced in lemma~\ref{geometry lemma 3}, we obtain
\begin{equation*}
N_{\mathbf{f}} \ll \sum_{\mathbf{u} \in \mathcal{T}_{\mathbf{f}}(U_0, U_1, U_2)}
\left( \frac1{|f_0 f_1 f_2|^{1/3}} \frac{(V_0 V_1 V_2)^{1/3}}{|u_0 u_1 u_2|^{1/3}} + 1 \right) \tau(|f_0 f_1 f_2 u_0 u_1 u_2|),
\end{equation*}
This implies in particular that
\begin{equation*}
N_{\mathbf{f}} \ll |f_0 f_1 f_2|^{\varepsilon}
\sum_{\mathbf{u} \in \mathcal{T}_{\mathbf{f}}(U_0, U_1, U_2)}
\left( \frac{(V_0 V_1 V_2)^{1/3}}{|u_0 u_1 u_2|^{1/3}} + 1 \right) \tau(|u_0 u_1 u_2|).
\end{equation*}
Let us write $u_i = z_i^2 \ell_i$ with $z_i \in \mathbb{Z}_{>0}$ and $|\mu(|\ell_i|)| = 1$ for $i \in \{0, 1, 2 \}$, and let
us set $\mathbf{l} = (\ell_0, \ell_1, \ell_2)$, $\mathbf{g} = (f_0 z_0^2, f_1 z_1^2,f_2 z_2^2)$ and $L_i = U_i/z_i^2$ for
$i \in \{0, 1, 2 \}$. We have
\begin{equation*}
N_{\mathbf{f}} \ll |f_0 f_1 f_2|^{\varepsilon} \! \sum_{\substack{z_i \leq U_i^{1/2} \\ i \in \{ 0, 1, 2 \}}}
|z_0 z_1 z_2|^{\varepsilon} \!
\sum_{\substack{\mathbf{l} \in \mathcal{T}_{\mathbf{g}}(L_0, L_1, L_2) \\ |\mu(|\ell_0 \ell_1 \ell_2|)| = 1}}
\left( \frac{(V_0 V_1 V_2)^{1/3}}{(z_0 z_1 z_2)^{2/3} |\ell_0 \ell_1 \ell_2|^{1/3}} + 1 \right) 2^{\omega(|\ell_0 \ell_1 \ell_2|)}.
\end{equation*}
Note that we have used the fact that $\ell_0 \ell_1 \ell_2$ is squarefree to replace the arithmetic function $\tau$ by $2^{\omega}$.
Let $\varepsilon > 0$ be fixed. We note that $\mathbf{g}$ is primitive so we can use lemma \ref{geometry lemma 3}. Thus, applying
partial summation and lemma~\ref{geometry lemma 3}, we get
\begin{equation*}
N_{\mathbf{f}} \ll |f_0 f_1 f_2|^{2 \varepsilon} \sum_{\substack{z_i \leq U_i^{1/2} \\ i \in \{ 0, 1, 2 \}}}
\frac{(U_0 U_1 U_2)^{2/3} (V_0 V_1 V_2)^{1/3} + U_0 U_1 U_2}{(z_0 z_1 z_2)^{2 - 3 \varepsilon}} M_{\varepsilon}(L_0, L_1, L_2).
\end{equation*}
This finally gives us a second upper bound
\begin{equation}
\label{bound 2}
N_{\mathbf{f}} \ll |f_0 f_1 f_2|^{2 \varepsilon} \left( (U_0 U_1 U_2)^{2/3} (V_0 V_1 V_2)^{1/3} + U_0 U_1 U_2 \right)
M_{2 \varepsilon}(U_0, U_1, U_2).
\end{equation}
As a result, putting together the upper bounds \eqref{bound 1} and \eqref{bound 2}, we find in particular that
\begin{equation*}
N_{\mathbf{f}} \ll |f_0 f_1 f_2|^{\varepsilon}
\left( (U_0 U_1 U_2)^{2/3} (V_0 V_1 V_2)^{1/3} + \min \{ U_0 U_1 U_2, V_0 V_1 V_2 \} \right) M_{\varepsilon}(U_0, U_1, U_2).
\end{equation*}
The simple observation that
\begin{equation*}
\min \{ U_0 U_1 U_2, V_0 V_1 V_2 \} \leq (U_0 U_1 U_2)^{2/3} (V_0 V_1 V_2)^{1/3},
\end{equation*}
completes the proof.
\end{proof}

\section{The lower bound}

\label{lower bound}

This section is devoted to the proof of the lower bound in Theorem \ref{Main Theorem}. As stated in the introduction,
the proof merely draws upon the fact that the contribution to $N_{U_{2}^2, \mathbf{H}}(B)$ of the
$\mathbf{y} \in \mathbb{P}^2(\mathbb{Q})$ whose height is bounded by a small power of $B$ is already of the expected order of
magnitude.

By definition of $N_{U_{2}^2, \mathbf{H}}(B)$, we have
\begin{equation*}
N_{U_{2}^2, \mathbf{H}}(B) = 2 \# \left\{ (\mathbf{x}, \mathbf{y}) \in \mathbb{Z}_{\neq 0}^3 \times \mathbb{Z}_{> 0}^3,
\begin{array}{l}
x_0 y_0^2 + x_1 y_1^2 + x_2 y_2^2 = 0 \\
\gcd(x_0, x_1, x_2) = \gcd(y_0, y_1, y_2) = 1 \\
\max_{i, j \in \{0, 1, 2 \}} x_i^2 y_j \leq B
\end{array}
\right\},
\end{equation*}
so that
\begin{equation*}
N_{U_{2}^2, \mathbf{H}}(B) \geq 12
\sum_{\substack{\mathbf{y} \in \mathbb{Z}_{> 0}^3 \\ \gcd(y_0, y_1, y_2) = 1 \\ y_0 < y_1 < y_2 \leq B^{1/6}}}
\# \left\{ \mathbf{x} \in \mathbb{Z}_{\neq 0}^3,
\begin{array}{l}
x_0 y_0^2 + x_1 y_1^2 + x_2 y_2^2 = 0 \\
\gcd(x_0, x_1, x_2) = 1 \\
\max_{i \in \{0, 1, 2 \}} x_i^2 y_2 \leq B
\end{array}
\right\}.
\end{equation*}
The condition $\gcd(y_0, y_2) = 1$ will be easier to handle than $\gcd(y_0, y_1, y_2) = 1$, so it is convenient to note that we also
have
\begin{equation*}
N_{U_{2}^2, \mathbf{H}}(B) \geq 12
\sum_{\substack{\mathbf{y} \in \mathbb{Z}_{> 0}^3 \\ \gcd(y_0, y_2) = 1 \\ y_0 < y_1 < y_2 \leq B^{1/6}}}
\# \left\{ \mathbf{x} \in \mathbb{Z}_{\neq 0}^3,
\begin{array}{l}
x_0 y_0^2 + x_1 y_1^2 + x_2 y_2^2 = 0 \\
\gcd(x_0, x_1, x_2) = 1 \\
\max_{i \in \{0, 1, 2 \}} x_i^2 y_2 \leq B
\end{array}
\right\}.
\end{equation*}
Since the condition $\max_{i \in \{0, 1 \}} x_i^2 y_2 \leq B/4$ and the equation $x_0 y_0^2 + x_1 y_1^2 + x_2 y_2^2 = 0$ imply that
$\max_{i \in \{0, 1, 2 \}} x_i^2 y_2 \leq B$, we have
\begin{equation*}
N_{U_{2}^2, \mathbf{H}}(B) \geq 12
\sum_{\substack{\mathbf{y} \in \mathbb{Z}_{> 0}^3 \\ \gcd(y_0, y_2) = 1 \\ y_0 < y_1 < y_2 \leq B^{1/6}}}
\# \left\{ \mathbf{x} \in \mathbb{Z}_{\neq 0}^3,
\begin{array}{l}
x_0 y_0^2 + x_1 y_1^2 + x_2 y_2^2 = 0 \\
\gcd(x_0, x_1, x_2) = 1 \\
\max_{i \in \{0, 1 \}} x_i^2 y_2 \leq B/4
\end{array}
\right\}.
\end{equation*}
We can now remove the coprimality condition $\gcd(x_0, x_1, x_2) = 1$ using a Möbius inversion. We get
\begin{equation}
\label{intermediate lower bound}
N_{U_{2}^2, \mathbf{H}}(B) \geq 12
\sum_{\substack{\mathbf{y} \in \mathbb{Z}_{> 0}^3 \\ \gcd(y_0, y_2) = 1 \\ y_0 < y_1 < y_2 \leq B^{1/6}}}
\sum_{k \leq B^{1/2}} \mu(k) \mathcal{S}_k(\mathbf{y}; B),
\end{equation}
where
\begin{equation*}
\mathcal{S}_k(\mathbf{y}; B) =
\# \left\{ \mathbf{x}' \in \mathbb{Z}_{\neq 0}^3,
\begin{array}{l}
x_0' y_0^2 + x_1' y_1^2 + x_2' y_2^2 = 0 \\
\max_{i \in \{0, 1 \}} x_i'^2 y_2 \leq B/4k^2
\end{array}
\right\},
\end{equation*}
and where we have used the obvious notation $\mathbf{x}' = (x_0', x_1', x_2')$. We now observe that
\begin{equation*}
\mathcal{S}_k(\mathbf{y}; B) =
\# \left\{ (x_0', x_1') \in \mathbb{Z}_{\neq 0}^2,
\begin{array}{l}
x_0' y_0^2  + x_1' y_1^2 = 0 \imod{y_2^2} \\
\max_{i \in \{0, 1 \}} x_i'^2 y_2 \leq B/4k^2
\end{array}
\right\}.
\end{equation*}
Since $\gcd(y_0, y_2) = 1$, $y_0$ is invertible modulo $y_2^2$. Using the notation $y_0^{-1}$ to denote the inverse of $y_0$ modulo
$y_2^2$, we have
\begin{align*}
\mathcal{S}_k(\mathbf{y}; B) & = \sum_{\substack{x_1' \in \mathbb{Z}_{\neq 0} \\ x_1'^2 y_2 \leq B/ 4 k^2}}
\# \left\{ x_0' \in \mathbb{Z}_{\neq 0},
\begin{array}{l}
x_0' = - y_0^{-2} x_1' y_1^2 \imod{y_2^2} \\
x_0'^2 y_2 \leq B/4k^2
\end{array}
\right\} \\
& = \sum_{\substack{x_1' \in \mathbb{Z}_{\neq 0} \\ x_1'^2 y_2 \leq B/ 4 k^2}}
\left( \frac{B^{1/2}}{k y_2^{5/2}} + O (1) \right) \\
& = \frac{B}{k^2 y_2^3} + O \left( \frac{B^{1/2}}{k y_2^{1/2}} \right).
\end{align*}
Recalling the lower bound \eqref{intermediate lower bound}, we see that we have obtained
\begin{equation*}
N_{U_{2}^2, \mathbf{H}}(B) \geq 12
\sum_{\substack{\mathbf{y} \in \mathbb{Z}_{> 0}^3 \\ \gcd(y_0, y_2) = 1 \\ y_0 < y_1 < y_2 \leq B^{1/6}}}
\sum_{k \leq B^{1/2}} \mu(k) \left( \frac{B}{k^2 y_2^3} + O \left( \frac{B^{1/2}}{k y_2^{1/2}} \right) \right).
\end{equation*}
This eventually gives
\begin{equation*}
N_{U_{2}^2, \mathbf{H}}(B) \geq \frac{B \log B}{\zeta(2)^2} + O(B),
\end{equation*}
which completes the proof of the lower bound in Theorem \ref{Main Theorem}.

\section{The upper bound}

\label{upper bound}

This section is concerned with establishing the upper bound in Theorem \ref{Main Theorem}. As already explained in the introduction,
the proof draws upon lemma \ref{geometry lemma 4}.

\subsection{Parametrization of the variables}

The following lemma provides us with a convenient parametrization of the rational points on $U_{2}^2$.

\begin{lemma}
\label{parametrization}
Let $\mathcal{T}(B)$ be the number of
$(f_0, f_1, f_2, g_0, g_1, g_2, h_0, h_1, h_2) \in \mathbb{Z}_{> 0}^9$ and $(u_0, u_1, u_2, v_0, v_1, v_2) \in \mathbb{Z}_{\neq 0}^6$
satisfying the equation
\begin{equation*}
f_0 u_0 v_0^2 + f_1 u_1 v_1^2 + f_2 u_2 v_2^2 = 0,
\end{equation*}
and the conditions $\gcd(f_i, f_j g_j h_i u_i v_j) = \gcd(g_i, g_j h_i u_i v_i v_j) = \gcd(h_i, h_j v_i) = 1$ and
$\gcd(u_i, u_j) = \gcd(v_i, v_j) = 1$ for $i, j \in \{0, 1, 2 \}$, $i \neq j$, and the height conditions
\begin{equation*}
\left( \max_{ \{ i, j, k \} =  \{ 0, 1, 2 \}} f_j f_k g_j^2 g_k^2 h_i^2 |u_i| \right)^2
\left( \max_{ \{ i, j, k \} =  \{ 0, 1, 2 \}} f_i g_i h_j h_k |v_i| \right) \leq B.
\end{equation*}
We have the equality
\begin{equation*}
N_{U_{2}^2, \mathbf{H}}(B) = \frac1{4} \mathcal{T}(B).
\end{equation*}
\end{lemma}

\begin{proof}
We have
\begin{equation*}
N_{U_{2}^2, \mathbf{H}}(B) = \frac1{4} \# \left\{ (\mathbf{x}, \mathbf{y}) \in \mathbb{Z}_{\neq 0}^3 \times \mathbb{Z}_{\neq 0}^3,
\begin{array}{l}
x_0 y_0^2 + x_1 y_1^2 + x_2 y_2^2 = 0 \\
\gcd(x_0, x_1, x_2) = \gcd(y_0, y_1, y_2) = 1 \\
\max_{i, j \in \{0, 1, 2 \}} x_i^2 |y_j| \leq B
\end{array}
\right\}.
\end{equation*}
For $\{ i, j, k \} =  \{ 0, 1, 2 \}$, let us set $h_i = \gcd(y_j,y_k)$ and let us write $y_i = h_j h_k y_i'$. The equation
\begin{equation*}
x_0 h_1^2 h_2^2 y_0'^2 + x_1 h_0^2 h_2^2 y_1'^2 + x_2 h_0^2 h_1^2 y_2'^2 = 0,
\end{equation*}
implies that for $i \in \{0, 1, 2 \}$, we have $h_i^2 \mid x_i$ so that we can write $x_i = h_i^2 x_i'$. We thus get the equation
\begin{equation*}
x_0' y_0'^2 + x_1' y_1'^2 + x_2' y_2'^2 = 0.
\end{equation*}
For $\{ i, j, k \} =  \{ 0, 1, 2 \}$, let us set $X_i = \gcd(x_j',x_k')$ and let us write $x_i' = X_j X_k u_i$. We get
\begin{equation*}
X_1 X_2 u_0 y_0'^2 + X_0 X_2 u_1 y_1'^2 + X_0 X_1 u_2 y_2'^2 = 0,
\end{equation*}
so that, for $i \in \{0, 1, 2 \}$, we have $X_i \mid y_i'^2$. As a consequence, for $i \in \{0, 1, 2 \}$, there is a unique way to write
$X_i = f_i g_i^2$ and $y_i' = f_i g_i v_i$ for $f_i, g_i \in \mathbb{Z}_{> 0}$ with $\gcd(g_i, v_i) = 1$. Therefore, we obtain the
equation
\begin{equation*}
f_0 u_0 v_0^2 + f_1 u_1 v_1^2 + f_2 u_2 v_2^2 = 0,
\end{equation*}
and it is not hard to check that the variables satisfy the coprimality conditions listed in the statement of the lemma, which
completes the proof.
\end{proof}

\subsection{Proof of the upper bound}

First, we note that the coprimality conditions $\gcd(f_i, v_j) = \gcd(u_i, u_j) = \gcd(v_i, v_j) = 1$ for $i, j \in \{0, 1, 2 \}$,
$i \neq j$, and the equation
\begin{equation*}
f_0 u_0 v_0^2 + f_1 u_1 v_1^2 + f_2 u_2 v_2^2 = 0,
\end{equation*}
imply that we actually have
$\gcd(u_i v_i, u_j v_j) = 1$ for $i, j \in \{0, 1, 2 \}$, $i \neq j$.

For $i \in \{ 0, 1, 2 \}$, let $F_i, G_i, H_i, U_i, V_i \geq 1$ run over powers of $2$ and let $\mathcal{M}$ be the number of
$(f_0, f_1, f_2, g_0, g_1, g_2, h_0, h_1, h_2) \in \mathbb{Z}_{> 0}^9$ and $(u_0, u_1, u_2, v_0, v_1, v_2) \in \mathbb{Z}_{\neq 0}^6$
satisfying the equation
\begin{equation*}
f_0 u_0 v_0^2 + f_1 u_1 v_1^2 + f_2 u_2 v_2^2 = 0,
\end{equation*}
the conditions $F_i < f_i \leq 2 F_i$, $G_i < g_i \leq 2 G_i$, $H_i < h_i \leq 2 H_i$, $U_i < |u_i| \leq 2 U_i$ and
$V_i < |v_i| \leq 2 V_i$, and $\gcd(f_i, f_j) = \gcd(u_i v_i, u_j v_j) = 1$ for $i, j \in \{ 0, 1, 2 \}$, $i \neq j$.
By lemma \ref{parametrization}, we have
\begin{equation*}
N_{U_{2}^2, \mathbf{H}}(B) \ll \sum_{\substack{F_i, G_i, H_i, U_i, V_i \\ i \in \{ 0, 1, 2 \}}} \mathcal{M},
\end{equation*}
where the sum is taken over the $F_i, G_i, H_i, U_i, V_i$, $i \in \{ 0, 1, 2 \}$, satisfying
\begin{equation}
\label{height conditions}
\left( \max_{ \{ i, j, k \} =  \{ 0, 1, 2 \}} F_j F_k G_j^2 G_k^2 H_i^2 U_i \right)^2
\left( \max_{ \{ i, j, k \} =  \{ 0, 1, 2 \}} F_i G_i H_j H_k V_i \right) \leq B.
\end{equation}
By choosing $\varepsilon = 1/6$ in lemma \ref{geometry lemma 4}, we get
\begin{equation*}
\mathcal{M} \ll (F_0 F_1 F_2)^{7/6} G_0 G_1 G_2 H_0 H_1 H_2 (U_0 U_1 U_2)^{2/3} (V_0 V_1 V_2)^{1/3}
M_{1/6}(U_0, U_1, U_2).
\end{equation*}
Recalling the definition of $M_{1/6}(U_0, U_1, U_2)$ given in lemma \ref{geometry lemma 3}, we define
\begin{equation*}
\mathcal{M}_1 = (F_0 F_1 F_2)^{7/6} G_0 G_1 G_2 H_0 H_1 H_2 (U_0 U_1 U_2)^{2/3} (V_0 V_1 V_2)^{1/3},
\end{equation*}
and
\begin{align*}
\mathcal{M}_2 = & \ (\log B) (F_0 F_1 F_2)^{7/6} G_0 G_1 G_2 H_0 H_1 H_2 (U_0 U_1 U_2)^{2/3} (V_0 V_1 V_2)^{1/3} \\
& \times \left( \min_{i, j \in \{ 0, 1, 2 \}, i \neq j} U_i U_j \right)^{-1/3},
\end{align*}
and also
\begin{equation*}
\mathcal{N}_{\ell}(B) = \sum_{\substack{F_i, G_i, H_i, U_i, V_i \\ i \in \{ 0, 1, 2 \}}} \mathcal{M}_{\ell},
\end{equation*}
for $\ell \in \{ 1, 2 \}$, and where the sum is taken over the $F_i, G_i, H_i, U_i, V_i$, $i \in \{ 0, 1, 2 \}$, satisfying the
conditions \eqref{height conditions}. We thus have
\begin{equation}
\label{inequality N}
N_{U_{2}^2, \mathbf{H}}(B) \ll \mathcal{N}_1(B) + \mathcal{N}_2(B).
\end{equation}
Let us start by taking care of $\mathcal{N}_1(B)$. For this, let us sum over $V_0$, $V_1$ and $V_2$ using the conditions
\eqref{height conditions}. We get
\begin{align*}
\mathcal{N}_1(B) \ll & \ B \sum_{\substack{F_i, G_i, H_i, U_i \\ i \in \{ 0, 1, 2 \}}}
(F_0 F_1 F_2)^{5/6} (G_0 G_1 G_2)^{2/3} (H_0 H_1 H_2)^{1/3} (U_0 U_1 U_2)^{2/3} \\
& \times \left( \max_{ \{ i, j, k \} =  \{ 0, 1, 2 \}} F_j F_k G_j^2 G_k^2 H_i^2 U_i \right)^{-2}.
\end{align*}
By symmetry, we can assume that
\begin{equation}
\label{height conditions U}
\max \left\{ F_0 F_2 G_0^2 G_2^2 H_1^2 U_1, F_0 F_1 G_0^2 G_1^2 H_2^2 U_2 \right\} \leq F_1 F_2 G_1^2 G_2^2 H_0^2 U_0.
\end{equation}
Let us sum over $U_1$ and $U_2$ using the inequalities \eqref{height conditions U}. We obtain
\begin{equation*}
\mathcal{N}_1(B) \ll B \sum_{\substack{F_i, G_i, H_i, U_0 \\ i \in \{ 0, 1, 2 \}}}
(F_0 F_1 F_2)^{-1/2} (G_0 G_1 G_2)^{-2} (H_0 H_1 H_2)^{-1},
\end{equation*}
which finally gives
\begin{equation}
\label{N1}
\mathcal{N}_1(B) \ll B \log B.
\end{equation}
Let us now deal with $\mathcal{N}_2(B)$. We can assume by symmetry that
\begin{equation*}
\min_{i, j \in \{ 0, 1, 2 \}, i \neq j} U_i U_j = U_1 U_2.
\end{equation*}
We thus have
\begin{align*}
\mathcal{M}_2 \ll & \ (\log B) (F_0 F_1 F_2)^{7/6} G_0 G_1 G_2 H_0 H_1 H_2 U_0^{2/3} (U_1 U_2)^{1/3}
(V_0 V_1 V_2)^{1/3}.
\end{align*}
Once again, let us sum over $V_0$, $V_1$ and $V_2$ using the conditions \eqref{height conditions}. We find that
\begin{align*}
\mathcal{N}_2(B) \ll & \ B (\log B) \sum_{\substack{F_i, G_i, H_i, U_i \\ i \in \{ 0, 1, 2 \}}}
(F_0 F_1 F_2)^{5/6} (G_0 G_1 G_2)^{2/3} (H_0 H_1 H_2)^{1/3} U_0^{2/3} (U_1 U_2)^{1/3} \\
& \times \left( \max_{ \{ i, j, k \} =  \{ 0, 1, 2 \}} F_j F_k G_j^2 G_k^2 H_i^2 U_i \right)^{-2}.
\end{align*}
Now, let us use the inequality
\begin{equation*}
\left( \max_{ \{ i, j, k \} =  \{ 0, 1, 2 \}} F_j F_k G_j^2 G_k^2 H_i^2 U_i \right)^{2} \geq
F_0 (F_1 F_2)^{3/2} G_0^2 (G_1 G_2)^{3} H_0^2 H_1 H_2 U_0 (U_1 U_2)^{1/2}.
\end{equation*}
This gives us
\begin{equation*}
\mathcal{N}_2(B) \ll B (\log B) \sum_{\substack{F_i, G_i, H_i, U_i \\ i \in \{ 0, 1, 2 \}}}
\left( F_0 (F_1 F_2)^{4} G_0^{8} (G_1 G_2)^{14} H_0^{10} (H_1 H_2)^4 U_0^2 U_1 U_2 \right)^{-1/6},
\end{equation*}
and therefore, we obtain
\begin{equation}
\label{N2}
\mathcal{N}_2(B) \ll B \log B.
\end{equation}
Putting together the three upper bounds \eqref{inequality N}, \eqref{N1} and \eqref{N2} completes the proof of the upper bound
in Theorem \ref{Main Theorem}.

\bibliographystyle{amsalpha}
\bibliography{biblio}

\end{document}